\documentclass{amsart}

\usepackage{amsmath}
\usepackage{amsthm} 
\usepackage{amsfonts}
\usepackage{amssymb}

\newcommand\R{{\mathbb{R}}}
\newcommand\C{{\mathbb{C}}}

\renewcommand\P{{\mathbf{P}}}
\newcommand\E{{\mathbf{E}}}

\newcommand\Var{\mathbf{Var}}
\renewcommand\Im{{\operatorname{Im}}}
\renewcommand\Re{{\operatorname{Re}}}
\newcommand\eps{{\varepsilon}}

\newcommand\tr{\operatorname{trace}}

\renewcommand\sc{{\operatorname{sc}}}
\newcommand\Arg{{\operatorname{Arg}}}

\newcommand\condo{{{\bf C0}}}

\subjclass{15A52}

\theoremstyle{plain}
  \newtheorem{theorem}{Theorem}

  \newtheorem{proposition}[theorem]{Proposition}

  \newtheorem{corollary}[theorem]{Corollary}

\theoremstyle{definition}
  \newtheorem{definition}[theorem]{Definition}
  
  \newtheorem{remark}[theorem]{Remark}

\include{psfig}

\begin{document}

\title[Concentration of eigenvalues]{Random matrices: \\ Sharp concentration  of eigenvalues}

\author{Terence Tao}
\address{Department of Mathematics, UCLA, Los Angeles CA 90095-1555}
\email{tao@math.ucla.edu}
\thanks{T. Tao is supported by a grant from the MacArthur Foundation, and by NSF grant DMS-0649473.}

\author{Van Vu}
\address{Department of Mathematics, Yale University, New Haven, CT }
\email{van.vu@yale.edu}
\thanks{V. Vu is supported by research grants DMS-0901216 and AFOSAR-FA-9550-09-1-0167.}

\begin{abstract}  Let $W_n= \frac{1}{\sqrt n}  M_n$ be a Wigner matrix
whose entries have vanishing third moment, normalized so that the spectrum is concentrated in the interval $[-2,2]$. We prove a  concentration bound for  $N_I = N_I(W_n)$,  the number of eigenvalues of $W_n$ in an interval $I$.

Our result shows that $N_I$  decays exponentially with standard deviation 
at most $O(\log^{O(1)} n)$. This is best possible up to the
constant  exponent in the logarithmic term.  As a corollary, the bulk eigenvalues are localized to an interval of width $O(\log^{O(1)} n/n)$; again, this is optimal up to the exponent. These results strengthen recent results of Erd\H os, Yau and Yin  (under the extra assumption of vanishing third moment). 

Our proof is relatively simple and relies on the Lindeberg replacement argument.
\end{abstract}

\maketitle

\setcounter{tocdepth}{2}

\section{Introduction}

The purpose of this paper is to sharpen the existing bounds on the eigenvalue counting function $N_I = N_I(W_n)$ of a (normalized) Wigner matrix $W_n = \frac{1}{\sqrt{n}} M_n$, and related quantities such as the Stieltjes transform $s_{W_n}(z)$ and individual eigenvalues $\lambda_i(W_n)$.  Let us first state the Wigner random matrix model which we will use.

\begin{definition}[Wigner matrices]\label{def:Wignermatrix}  Let $n \geq 1$ be an integer (which we view as a parameter going off to infinity; in particular, $n$ is understood to be large enough that quantities such as $\log \log n$ are well-defined and positive). An $n \times n$ \emph{Wigner matrix} $M_n$ is defined to be a  random Hermitian $n \times n$ matrix $M_n = (\xi_{ij})_{1 \leq i,j \leq n}$, in which the $\xi_{ij}$ for $1 \leq i \leq j \leq n$ are jointly independent with $\xi_{ji} = \overline{\xi_{ij}}$ (in particular, the $\xi_{ii}$ are real-valued).  For $1 \leq i < j \leq n$, we require that the $\xi_{ij}$ have mean zero and variance one, while for $1 \leq i=j \leq n$ we require that the $\xi_{ij}$ (which are necessarily real) have mean zero and variance $\sigma^2$ for some $\sigma^2>0$ independent of $i,j,n$.   For simplicity, we will also assume that for each $1 \leq i < j \leq n$, the real and imaginary parts $\Re \xi_{ij}$, $\Im \xi_{ij}$ are independent.  We refer to the distributions $\Re \xi_{ij}$, $\Im \xi_{ij}$ for $1\leq i<j \leq n$  and $\xi_{ii}$ for $1 \leq i \leq n$ as the \emph{atom distributions} of $M_n$, and view them as fixed while $n$ goes off to infinity.  

We say that the Wigner matrix ensemble \emph{obeys Condition {\condo}} if we have the exponential decay condition
\begin{equation}\label{exp-decay}
\P(|\xi_{ij}|\ge t^C) \le e^{-t} 
\end{equation}
for all $1 \leq i,j \leq n$ and $t \ge C'$, and some constants $C, C'$ (independent of $i,j,n$).  

Two Wigner matrices $M_n = (\xi_{ij})_{1 \leq i,j \leq n}$ and $M'_n = (\xi'_{ij})_{1 \leq i,j \leq n}$ are said to have \emph{matching moments to order $m$} for some $m \geq 0$ if one has
\begin{equation}\label{match-1}
 \E \Re(\xi_{ij})^k \Im(\xi_{ij})^l = \E \Re(\xi'_{ij})^k \Im(\xi'_{ij})^l 
\end{equation}
for all $1 \leq i, j \leq n$ and all natural numbers $k,l \geq 0$ with $k+l \leq m$.  As we are assuming the real and imaginary parts to be independent, this condition simplifies to the conditions
\begin{equation}\label{match-2}
 \E \Re(\xi_{ij})^k = \E \Re(\xi'_{ij})^k; \quad \E \Im(\xi_{ij})^k = \E \Im(\xi'_{ij})^k
\end{equation}
for all $1 \leq i, j \leq n$ and all $0 \leq k \leq m$.  If we only require \eqref{match-1} or \eqref{match-2} to hold in the off-diagonal case $i \neq j$ (resp. in the diagonal case $i=j$), we say that $M_n$ and $M'_n$ \emph{match moments to order $m$ off the diagonal} (resp. on the diagonal).
\end{definition}

We observe four basic examples of Wigner matrices:
\begin{itemize}
\item In the \emph{Gaussian Unitary Ensemble} (GUE), $\xi_{ij} \equiv N(0,1)_\C$ is the standard complex gaussian random variable for $1 \leq i < j \leq n$, $\xi_{ii} \equiv N(0,1)_\R$ is the standard real gaussian random variable for $1 \leq i \leq n$, and $\sigma^2 = 1$.
\item In the \emph{Gaussian Orthogonal Ensemble} (GOE) $\xi_{ij} \equiv N(0,1)_\R$ is the standard real gaussian random variable for $1 \leq i < j \leq n$, $\xi_{ii} \equiv N(0,2)_\R$ is a slightly rescaled real gaussian random variable for $1 \leq i \leq n$, and $\sigma^2=2$.
\item In the \emph{symmetric Bernoulli ensemble}, $\xi_{ij}$ equals $+1$ with probability $1/2$ and $-1$ with probability $1/2$ for all $1 \leq i, j \leq n$, and $\sigma^2=1$.
\item In the \emph{complex Hermitian Bernoulli ensemble}, $\Re \xi_{ij}, \Im \xi_{ij}$ for $1 \leq i < j \leq n$ and $\xi_{ii}$ for $1 \leq i \leq n$ all equal $+1$ with probability $1/2$ and $-1$ with probability $1/2$, and $\sigma^2=1$.
\end{itemize}

\begin{remark}  Note that we do not require the off-diagonal $\xi_{ij}$, $1 \leq i < j \leq n$ (or the diagonal $\xi_{ii}$, $1 \leq i \leq n$) to be identically distributed.  This lack of an identical distribution hypothesis will be convenient when we apply the Lindeberg exchange strategy \cite{lindeberg}, in which one Wigner matrix is compared to another one by exchanging the entries of the former matrix with the latter one\footnote{More precisely, we exchange the diagonal entries one at a time, and the off-diagonal entries two at a time, in order to preserve the Hermitian property throughout.} at a time.  As such, the intermediate stages of this exchange process need not have identically distributed entries, even if the initial and final matrices do.

The hypothesis of independence of real and imaginary parts is imposed purely to simplify the exposition, and can easily be removed at the cost of some more complicated notation; in particular, the simpler moment matching condition \eqref{match-2} must be replaced by the more complicated condition \eqref{match-1}. See Remark \ref{non-indep}.
\end{remark}

In this paper, we will mostly deal with the (coarse-scale) normalization $W_n := \frac{1}{\sqrt n} M_n$ of $M_n$ of the Wigner matrix, and more specifically with the eigenvalue counting function
$$ N_I = N_I(W_n) := | \{ 1 \leq i \leq n: \lambda_i(W_n) \in I \} |$$
of this matrix for various intervals $I \subset \R$, where $\lambda_1(W_n) \leq \ldots \leq \lambda_n(W_n)$ denote the (necessarily) real eigenvalues of the (Hermitian) matrix $W_n$.

The well-known \emph{Wigner semicircle law} describes the bulk behavior of the counting function $N_I$ of a Wigner matrix in terms of the \emph{semicircular distribution} $\rho_{sc}(x)\ dx$, where $\rho_{sc}: \R \to \R$ is the function
$$ \rho_{sc}(x) := \frac{1}{2\pi} (4-x^2)_+^{1/2}.$$

\begin{theorem}[Semicircular law]\label{theorem:Wigner} Let $M_n$ be a Wigner Hermitian matrix obeying Condition {\condo}.  Then for any fixed interval $I$ (independent of $n$), one has
$$\lim_{n \rightarrow \infty} \frac{1}{n} N_I(W_n)
 = \int_I \rho_{sc}(y)\ dy$$
in the sense of probability.
\end{theorem}

See for instance \cite{bai} for a proof of this theorem and for historical background.  Condition {\condo} can be omitted from this law, but we retain the hypothesis as it will be needed for the subsequent results discussed below.

If we use $o(x)$ to denote a quantity that goes to zero as $n \to \infty$ after dividing by $x$, we can reformulate Theorem \ref{theorem:Wigner} as the assertion that the asymptotic
\begin{equation}\label{niwn}
 N_I(W_n) = n \int_I \rho_{sc}(y)\ dy + o(n)
\end{equation}
holds with probability $1-o(1)$ for each fixed $I$.

One can also phrase the semicircular law in terms of the individual eigenvalues $\lambda_i(W_n)$.  If for each $1 \leq i \leq n$ we define the \emph{classical location} $\gamma_i$ of the normalised $i^{\operatorname{th}}$ eigenvalue by the formula
\begin{equation}\label{gammai}
\int_{-\infty}^{\gamma_i} \rho_{sc} (x) dx = \frac{i}{n}. 
\end{equation}
then the Wigner semicircular law (combined with an almost sure bound of $(2+o(1))\sqrt{n}$ for the operator norm of $M_n$, due to Bai and Yin \cite{baiyin}) is equivalent to the assertion that one has
\begin{equation}\label{gammai-form}
\lambda_i(W_n) = \gamma_i + o(1)
\end{equation}
for any given $1 \leq i \leq n$, with probability $1-o(1)$.

In this paper we investigate sharper versions of the semicircular law (known in the literature as \emph{local semicircular laws}), which improve upon the error terms and failure probabilities in \eqref{niwn} and \eqref{gammai-form}, and in which the interval $I$ is now allowed to depend on $n$.

We first discuss the case of the Gaussian Unitary Ensemble (GUE), which is the most well-understood case, as the joint distribution of the eigenvalues is given by a determinantal point process.  Because of this, it is known that for any interval $I$, the random variable $N_I(W_n)$ in the GUE case obeys a law of the form
\begin{equation}\label{etain}
 N_I(W_n) \equiv \sum_{i=1}^{\infty} \eta_i
\end{equation}
where the $\eta_i = \eta_{i,n,I}$ are jointly independent indicator random variables (i.e. they take values in $\{0,1\}$); see e.g. \cite[Corollary 4.2.24]{AGZ}.  The mean and variance of $N_I(W_n)$ can also be computed in the GUE case with a high degree of accuracy:

\begin{theorem}[Mean and variance for GUE]\label{mevar}  Let $M_n$ be drawn from GUE, let $W_n := \frac{1}{\sqrt{n}} M_n$, and let $I = [-\infty,x]$ for some real number $x$ (which may depend on $n$).  Let $\eps>0$ be independent of $n$.
\begin{itemize}
\item[(i)] (Bulk case) If $x \in [-2+\eps, 2-\eps]$, then
$$ \E N_I(W_n) = n \int_I \rho_\sc(y)\ dy + O\left( \frac{\log n}{n} \right).$$
\item[(ii)] (Edge case) If $x \in [-2,2]$, then
$$ \E N_I(W_n) = n \int_I \rho_\sc(y)\ dy + O(1).$$
\item[(iii)] (Variance bound) If one has $x \in [-2,2-\eps]$ and $n^{2/3} (2+x) \to \infty$ as $n \to \infty$, one has
$$ \Var N_I(W_n) = \left(\frac{1}{2\pi^2} + o(1)\right) \log (n (2+x)^{3/2}).$$
In particular, one has $\Var N_I(W_n) = O( \log n )$ in this regime.
\end{itemize}
\end{theorem}

Here of course we use $X=O(Y)$, $X \ll Y$ or $Y \gg X$ to denote the estimate $|X| \leq CY$ for some quantity $C$ independent of $n$.  We will also use $c$ to denote various small positive constants $c>0$ independent of $n$ (but possibly depending on the constants in Condition {\condo}).

\begin{proof} See \cite[Lemmas 2.1, 2.2, 2.3]{Gus}.  Note that the normalization conventions in \cite{Gus} differ by a factor of $\sqrt{2}$ from the ones used here\footnote{There is a slight inaccuracy in the statement of \cite[Lemma 2.2]{Gus}, namely that the main term of $\frac{4\sqrt{2}}{3\pi} n(1-t)^{3/2}$ in that lemma should be replaced with the more accurate main term $\frac{2n}{\pi} \int_t^1 \sqrt{1-x^2}\ dx$ (which is what actually comes out of the proof of \cite[Lemma 2.2]{Gus}).  These two main terms differ by $O(1)$ in the regime $t = 1 - O(n^{-2/5})$ as can be seen from a Taylor expansion, but they differ by more than $O(1)$ outside of this regime.}.
\end{proof}

By combining these estimates with a well-known inequality of Bennett \cite{Ben62}, we obtain a concentration estimate for $N_I(W_n)$ in the GUE case:

\begin{corollary}[Concentration for GUE]\label{conc} Let $M_n$ be drawn from GUE, let $W_n := \frac{1}{\sqrt{n}} M_n$, and let $I$ be an interval.  Then one has
$$ \P( |N_I(W_n) - n \int_I \rho_\sc(y)\ dy| \geq T) \ll \exp( - c T )$$
for all $T \gg \log n$.
\end{corollary}

\begin{proof}  By the triangle inequality we may take $I = [-\infty,x]$ for some real number $x$.  As $\rho_\sc$ is supported on $[-2,2]$ and has total mass $1$, we see (using the trivial bounds $0 \leq N_I(W_n) \leq n$ and $N_I(W_n) \leq N_J(W_n)$ whenever $I \subset J$) that without loss of generality we may assume $x \in [-2,2]$.  By \eqref{etain} and Theorem \ref{mevar}, $N_I(W_n)$ is then the sum of independent indicator functions, and the mean $\mu$ and variance $\sigma^2$ of this sum is given by
$$ \mu = n \int_I \rho_\sc(y)\ dy + O(1)$$
and $\sigma^2 = O(\log n)$ respectively.  Bennett's inequality (see \cite{Ben62}, or \cite[p.29]{JLR}) then asserts that
$$ \P( |N_I(W_n) - \mu| \geq t ) \leq 2 \exp( - \sigma^2 \phi(\frac{t}{\sigma^2}) )$$
where $\phi(x) := (1+x) \log(1+x)-x$.  Since $\phi(x) \gg x$ when $x \gg 1$, the claim follows\footnote{Indeed, this argument shows a slightly better bound than $\exp(-cT)$.  One can also use Bernstein's inequality to also obtain the $\exp(-cT)$ bound if desired.}.  
\end{proof}

Let us say that an event holds with \emph{overwhelming probability} if it occurs with probability $1-O(n^{-A})$ for each fixed $A$.  From the above corollary we see in particular that in the GUE case, one has
$$ N_I(W_n) = n \int_I \rho_\sc(y)\ dy + O(\log^{1+o(1)} n)$$
with overwhelming probability for each fixed $I$, and an easy union bound argument (ranging over all intervals $I$ in, say, $[-3,3]$ whose endpoints are a multiple of $n^{-100}$ (say)) then shows that this is also true uniformly in $I$ as well.

\begin{remark} By using a general result of Costin and Lebowitz \cite{CL}, one can also obtain a central limit theorem for $N_I(W_n)$ as long as $I$ is not too small; see \cite{Gus}.  Such results have also been recently been extended to more general Wigner matrices in \cite{dall}.  However, such theorems will not be the focus of the current paper.
\end{remark}

Now we turn from the GUE case to more general Wigner ensembles.  There has been much interest in recent years in obtaining concentration results for $N_I(W_n)$ (and for closely related objects, such as the Stieltjes transform $s_{W_n}(z) := \frac{1}{n} \tr (W_n - z)^{-1}$ of $W_n$) for short intervals $I$, due to the applicability of such results to establishing various universality properties of such matrices; see \cite{ESY1, ESY2, ESY3, TVlocal1, TVedge, ESY, EYY, EYY2}.  The previous best result in this direction was by Erd\H{o}s, Yau, and Yin \cite{EYY2} (see also \cite{ekyy} for a variant):

\begin{theorem}\label{ween}\cite{EYY2} Let $M_n$ be a Wigner matrix obeying Condition {\condo}, and let $W_n := \frac{1}{\sqrt{n}} M_n$.  Then, for any interval $I$, one has
\begin{equation}\label{niwna}
\P( |N_I(W_n) - n \int_I \rho_\sc(y)\ dy| \geq T) \ll \exp( - c T^c )
\end{equation}
for all $T \geq \log^{A \log \log n} n$, and some constant $A>0$.
\end{theorem}

\begin{proof}  See \cite[Theorem 2.2]{EYY2}.
\end{proof}

One can reformulate \eqref{niwna} equivalently as the assertion that
$$ \P( |N_I(W_n) - n \int_I \rho_\sc(y)\ dy| \geq T) \ll \exp(\log^{O(\log \log n)} n) \exp( - c T^c )$$
for all $T>0$.

In particular, this theorem asserts that with overwhelming probability one has
$$N_I(W_n) = n \int_I \rho_\sc(y)\ dy + O( \log^{O(\log\log n)} n )$$
for all intervals $I$.  The proof of the above theorem is somewhat lengthy, requiring a delicate analysis of the self-consistent equation of the Stieltjes transform of $W_n$.  A recent preprint of G\"otze and Tikhomirov \cite{gotze} has claimed\footnote{At the current time of writing, the preprint \cite{gotze} is being revised to address some gaps in the proofs of some lemmas in that paper, specifically Lemmas 5.2 and 5.3 from \cite{gotze} (private communication).} an improvement to this result, namely that
\begin{equation}\label{niww}
N_I(W_n) = n \int_I \rho_\sc(y)\ dy + O( \log^C n )
\end{equation}
with probability $1 - O(\exp(-c \log n (\log \log n)^\alpha))$ for certain explicit exponents $C,\alpha$.  This claim would imply as a consequence that for any interval $I$, $N_I(W_n)$ has variance $O(\log^{O(1)} n)$.  

Comparing Theorem \ref{ween} with the previous results for the GUE case, we see that there is a loss of a double logarithm $\log \log n$ in the exponent.  The first main result of this paper\footnote{We would like to thank M. Ledoux for a private conversation that led to this question.} is to remove this double logarithmic loss, at least under an additional vanishing moment assumption:

\begin{theorem}[First main theorem]\label{main} Let $M_n$ be a Wigner matrix obeying Condition \condo, and let $W_n := \frac{1}{\sqrt{n}} M_n$.  Assume that $M_n$ matches moments with GUE to third order off the diagonal (i.e. $\Re \xi_{ij}, \Im \xi_{ij}$ have variance $1/2$ and third moment zero).  Then, for any interval $I$, one has
$$
\P( |N_I(W_n) - n \int_I \rho_\sc(y)\ dy| \geq T) \ll n^{O(1)} \exp( - c T^c )$$
for any $T > 0$.
\end{theorem}

This estimate is phrased for any $T$, but the bound only becomes non-trivial when $T \gg \log^C n$ for some sufficiently large $C$.   In that regime, we see that this result removes the double-logarithmic factor from Theorem \ref{ween}; it is also comparable to the result \eqref{niww} from \cite{gotze} when $T = \log^{O(1)} n$ (though not with as sharp a set of exponents as \cite{gotze}, and one also needs an additional moment matching hypothesis), but gives additional large deviation bounds when $T$ is much larger than $\log^{O(1)} n$.

\begin{remark}  As we are assuming $\Re(\xi_{ij})$ and $\Im(\xi_{ij})$ to be independent, the moment matching condition simplifies to the constraints that $\E \Re(\xi_{ij})^2 = \E \Im(\xi_{ij})^2 = \frac{1}{2}$ and $\E \Re(\xi_{ij})^3 = \E \Im(\xi_{ij})^3 = 0$.  However, it is possible to extend this theorem to the case when the real and imaginary parts of $\xi_{ij}$ are not independent; see Remark \ref{non-indep}.
\end{remark}

\begin{remark}
The constant $c$ in the bound in  Theorem \ref{main} is quite decent in several cases. For instance, if the atom variables of $M_n$ are Bernoulli or have sub-gaussian tail, then we can set $c= 2/5-o(1)$ by optimizing our arguments (details omitted).  If we assume 4 matching moments rather than 3, then we can set $c=1$ (see Remark \ref{remark:bestT}), matching the bound in Corollary \ref{conc}.  It is an  interesting question to determine the best value of $c$. The value of $c$ in \cite{EYY} is implicit and rather small.
\end{remark} 

We prove Theorem \ref{main} in Sections \ref{stieltjes-sec}-\ref{stable-sec}.  Our argument differs from that in \cite{EYY2} in that it only uses a relatively crude analysis of the self-consistent equation to obtain some preliminary bounds on the Stieltjes transform and on $N_I$ (which were also essentially implicit in previous literature).  Instead, the bulk of the argument relies on using the Lindeberg swapping strategy to deduce concentration of $N_I(W_n)$ in the non-GUE case from the concentration results in the GUE case provided by Corollary \ref{conc}.  In order to keep the error terms in this swapping under control, three matching moments\footnote{Compare with the ``four moment theorem'' from \cite{TVlocal1}.  We need one less moment here because we are working at ``mesoscopic'' scales (in which the number of eigenvalues involved is much larger than $1$) rather than at ``microscopic'' scales.  However, in Theorem \ref{main-2} below, only one eigenvalue is involved, making the problem microscopic enough to require four moments instead of three.} are needed.

Very roughly speaking, the main idea of the argument is to show that high moments such as
$$ \E |N_I(W_n) - n \int_I \rho_\sc(y)\ dy|^k$$
are quite stable (in a multiplicative sense) if one swaps (the real or imaginary part of) one of the entries of $W_n$ (and its adjoint) with another random variable that matches the moments of the original entry to third order.  For technical reasons, however, we do not quite manipulate $N_I(W_n)$ directly, but instead work with a proxy for this quantity, namely a certain integral of the Stieltjes transform of $W_n$.  As observed in \cite{EYY}, the Lindeberg swapping argument is quite simple to implement at the level of the Stieltjes transform (due to the simplicity of the resolvent identities, when compared against the rather complicated Taylor expansions of individual eigenvalues used in \cite{TVlocal1}).

The result in Theorem \ref{main} is well suited for controlling eigenvalues in the bulk of the spectrum, but is not sufficient by itself to control eigenvalues at the edge, and in particular the largest eigenvalue $\lambda_1(W_n)$ and the smallest eigenvalue $\lambda_n(W_n)$.  However, it is known that these eigenvalues are highly concentrated around $+2$ and $-2$ respectively. In the GUE case, we have the following concentration result of Aubrun \cite{aubrun}:

\begin{theorem}[Concentration for GUE]\label{conc-extreme}\cite{aubrun} Let $M_n$ be drawn from GUE, let $W_n := \frac{1}{\sqrt{n}} M_n$.  Then one has
$$ \P( n^{2/3} (\lambda_1(W_n) - 2) \geq T ) \ll \exp( - c T^{3/2} )$$
for all $T > 0$.  By symmetry, we also have
$$ \P( n^{2/3}(-\lambda_n(W_n) - 2) \geq T ) \ll \exp( - c T^{3/2} ).$$
\end{theorem}

\begin{remark} As is well known, the random variable $n^{2/3} (\lambda_1(W_n)-2)$ in fact converges in distribution to the Tracy-Widom law \cite{TW}.  However, we will not focus on this law here.  The exponent $3/2$ on the right-hand side cannot be improved (indeed, it matches the decay rate of the Tracy-Widom law); see \cite{aubrun} for further discussion.
\end{remark}

This result was partially extended to the Wigner case in \cite{EYY2}:

\begin{theorem}\label{ween-extreme}\cite{EYY2} Let $M_n$ be a Wigner matrix obeying Condition \condo, and let $W_n := \frac{1}{\sqrt{n}} M_n$.  Then one has
\begin{equation}\label{ltc}
 \P( n^{2/3} (\lambda_1(W_n) - 2) \geq T ) \ll \exp( - c T^c )
\end{equation}
for all $T \geq \log^{A \log \log n} n$, for some $A >0$ independent of $n$.  By symmetry, one also has
$$ \P( n^{2/3} (-\lambda_n(W_n) - 2) \geq T ) \ll \exp( - c T^c ).$$
\end{theorem}

\begin{proof} See \cite[Theorem 2.1]{EYY2}.
\end{proof}

As before, we can reformulate \eqref{ltc} equivalently as the assertion that
$$ \P( n^{2/3} (\lambda_1(W_n) - 2) \geq T ) \ll \exp( \log^{O(\log \log n)} n) \exp( - c T^c )$$
for all $T>0$.

Our second main result is to remove the double logarithm from Theorem \ref{ween-extreme}, at the cost of requiring matching GUE to fourth order rather than to third order:

\begin{theorem}[Second main theorem]\label{main-2} Let $M_n$ be a Wigner matrix obeying Condition \condo, and let $W_n := \frac{1}{\sqrt{n}} M_n$.  Assume that $M_n$ matches moments with GUE to fourth order off the diagonal and second order on the diagonal (i.e. $\sigma^2=1$).  Then one has
$$ \P( n^{2/3} (\lambda_1(W_n) - 2) \geq T ) \ll n^{O(1)} \exp( - c T^c )$$
for any $T > 0$.  By symmetry, one then also has
$$ \P( n^{2/3} (-\lambda_n(W_n) - 2) \geq T ) \ll n^{O(1)} \exp( - c T^c )$$
\end{theorem}

We will derive Theorem \ref{main-2} from Theorem \ref{conc-extreme} in Section \ref{second-sec} using the same techniques used to derive Theorem \ref{main} from Corollary \ref{conc}.

By combining Theorem \ref{main} and Theorem \ref{main-2} one can ``solve'' for individual eigenvalues $\lambda_i(W_n)$ to obtain an appropriate concentration
(localization)  result:

\begin{corollary}[Concentration of eigenvalues] Let $M_n$ be a Wigner matrix obeying Condition \condo, and let $W_n := \frac{1}{\sqrt{n}} M_n$.  Assume that $M_n$ matches moments with GUE to fourth order off the diagonal and second order on the diagonal.  
Then for any $1 \leq i \leq n$, we have
$$ \P( n^{2/3} \min(i, n-i+1)^{1/3} |\lambda_i(W_n) - \gamma_i| \geq T ) \ll n^{O(1)} \exp(-cT^c)$$
for any $T>0$.

If we assume only three matching moments, then the above estimate still holds provided that we have the additional hypothesis
$$ \min(i, n+1-i) \geq T^{c'}$$
for some fixed $c'>0$ (where the constant $c$ above is allowed to depend on $c'$).
\end{corollary}

The second part of this corollary significantly improves \cite[Theorem 29]{TVlocal1}.  (As a matter of fact, the original proof 
of this theorem has a gap in it;  see \cite[Appendix A]{TVsurvey} for a further discussion.) 

\begin{proof} First assume four matching moments. By Theorems \ref{main}, \ref{main-2} and the union bound, we see that outside of an event of probability $n^{O(1)} \exp(-cT^c)$, we have
\begin{equation}\label{n1}
 N_I = n \int_I \rho_\sc(y)\ dy + O(T)
\end{equation}
for all intervals $I$, as well as the bounds
\begin{equation}\label{n2}
-2 - O(n^{-2/3} T) \leq \lambda_n(W_n) \leq \lambda_1(W_n) \leq 2 + O( n^{-2/3} T).
\end{equation}
Some elementary estimation of the semicircular density $\rho_\sc$ and its integrals $\int_I \rho_\sc(y)\ dy$ (cf. \cite[\S 5]{EYY2}) then gives
$$ \lambda_i(W_n) = \gamma_i + O( n^{-2/3} \min(i, n-i+1)^{-1/3} T )$$
for all $1 \leq i \leq n$.  The claim follows (possibly after adjusting $T$ by a multiplicative factor).

Now suppose we only have three matching moments.  Then by Theorem \ref{main} and the union bound, we may assume that
$$
 |N_I - n \int_I \rho_\sc(y)\ dy| < T^{c'}
$$
for all $I$. In particular (setting $I$ equal to $[2,+\infty)$ or $(-\infty,-2]$) this implies that $-2 \leq \lambda_i(W_n) \leq 2$ whenever $\min(i, n+1-i) \geq T^{c'}$.  One can then argue as before.
\end{proof}

\begin{remark}  The results in this paper also hold if one replaces the GUE ensemble by the GOE ensemble, in which case one considers real symmetric Wigner matrices instead of Hermitian Wigner matrices, with the off-diagonal $\xi_{ij}$ having mean zero, variance one, and third moment zero (if there are three matching moments) and fourth moment equal to $3$ (if there are four matching moments).  To do this, one needs to replace Theorem \ref{mevar} and Theorem \ref{conc-extreme} by their GOE counterparts.  The GOE version of Theorem \ref{mevar} was established by O'Rourke \cite{rourke}.
The GOE version of Theorem \ref{conc-extreme} follows from the results in \cite{LR}.  In principle, one might be able to use other ensembles (such as the gaussian divisible matrices \cite{Joh1}) to match moments with, which would allow one to remove the moment conditions almost entirely.  We will not pursue these matters here.
\end{remark}

We are indebted to the anonymous referees for several suggestions and corrections.

\section{Reduction to the Stieltjes transform}\label{stieltjes-sec}

We now begin the proof of Theorem \ref{main}.  The first step is to replace the counting function $N_I = N_I(W_n)$ with the \emph{Stieltjes transform} $s_{W_n}$, defined by the formula
\begin{equation}\label{swn-def}
s_{W_n}(z) := \frac{1}{n} \tr(W_n - z)^{-1} = \frac{1}{n} \sum_{i=1}^n \frac{1}{\lambda_i(W_n) - z}
\end{equation}
for any complex number $z$ with positive imaginary part.  We can express this Stieltjes transform as a Riemann-Stieltjes integral
\begin{equation}\label{sneeze}
 s_{W_n}(z) = \frac{1}{n} \int_\R \frac{1}{x-z}\ dN_{(-\infty,x)}.
\end{equation}
which gives a clear connection between the Stieltjes transform and the counting function; in the converse direction, we have the identity
\begin{equation}\label{bless}
\frac{\pi}{2} - \frac{\pi}{n} N_{(-\infty,E)} =  \Re \int_0^\infty s_{W_n}(E+\sqrt{-1}\eta)\ d\eta
\end{equation}
whenever $E$ is not an eigenvalue of $W_n$, showing that (in principle at least) we can reconstruct the eigenvalue counting function from the Stieltjes transform.

Using the heuristic $dN_{(-\infty,x)} \approx n \rho_\sc(x)\ dx$ from \eqref{niwn}, we thus expect from \eqref{sneeze} to have $s_{W_n} \approx s_\sc$, where
$$ s_\sc(z) := \int_\R \frac{1}{x-z} \rho_\sc(x)\ dx.$$
As is well known (see e.g. \cite{bai}), $s_\sc$ can be evaluated explicitly via contour integration\footnote{For instance, one can observe that $\frac{1}{\pi} \Im s_\sc(x \pm \sqrt{-1}\eps)$ converges to $\pm \rho_{sc}(x)$ as $\eps \to 0^+$, and then apply the Cauchy integral formula to $s_\sc$ around the slit $[-2,2]$.}\begin{equation}\label{explicit}
 s_\sc(z) = \frac{1}{2} (-z + \sqrt{z^2-4}),
\end{equation}
where $\sqrt{z^2-4}$ is the branch of the square root that is asymptotic to $z$ at infinity.  In particular, $s_\sc$ exactly obeys the \emph{self-consistent equation}
\begin{equation}\label{sce}
s_\sc(z) = -\frac{1}{s_\sc(z) + z}
\end{equation}

In the case of GUE, we may easily formalize this heuristic with the assistance of Corollary \ref{conc}:

\begin{proposition}[Concentration for GUE]\label{gack} Let $M_n$ be drawn from GUE, and $W_n := \frac{1}{\sqrt{n}} M_n$.  Then for any $T > 0$ and any complex number $z = E + \sqrt{-1}\eta$ with $\eta>0$, one has
$$ \P\left( |s_{W_n}(z) - s_\sc(z)| \geq \frac{T}{n\eta} \right) \ll n^{O(1)} \exp( - c T ).$$
\end{proposition}

\begin{proof}  We may assume that $T \gg \log n$, as the claim is trivial otherwise.
Let $T_1 \gg \log n$ be chosen later.  From Corollary \ref{conc} and the union bound, we see that with probability $1 - O( n^{O(1)} \exp(-cT_1) )$, one has
$$ \left|N_I(W_n) - n \int_I \rho_\sc(y)\ dy\right| \ll T_1$$
for all intervals $I$ in $[-3,3]$ whose endpoints are multiples of $n^{-100}$, and hence for all intervals $I$.  In particular,
$$ N_{(-\infty,x)} = n \int_{-\infty}^x \rho_\sc(y)\ dy + O(T_1)$$
for all $x$.  On the other hand, from \eqref{sneeze} and integration by parts, one has
$$ s_{W_n}(z) = \frac{1}{n} \int_\R \frac{1}{(x-z)^2} N_{(-\infty,x)}\ dx.$$
A similar integration by parts gives
$$ s_\sc(z) = \int_\R \frac{1}{(x-z)^2} \left(\int_{-\infty}^x \rho_\sc(y)\ dy\right)\ dx,$$
and thus by the triangle inequality
$$ s_{W_n}(z) = s_\sc(z) + O( \frac{1}{n} \int_\R \frac{1}{|x-z|^2} T_1\ dx ).$$
The error term on the right-hand side evaluates to $O(\frac{T_1}{n\eta})$.  The claim then follows by choosing $T_1$ to be a small multiple of $T$.
\end{proof}

We will use this proposition to obtain a similar concentration result for Wigner matrices:

\begin{theorem}[Concentration for Wigner]\label{wigo} Let $M_n$ be a Wigner matrix obeying Condition \condo, and let $W_n := \frac{1}{\sqrt{n}} M_n$.  Assume that $M_n$ matches moments with GUE to third order off the diagonal.  Then for any $T > 0$ and any complex number $z = E + \sqrt{-1}\eta$ with $E \in [-3,3]$ and $0 < \eta \ll n^{100}$, one has
$$ \P( |s_{W_n}(z) - s_\sc(z)| \geq \frac{T}{n\eta} ) \ll n^{O(1)} ( \exp( - c T^c ) + \exp( - c (n\eta)^c ) ).$$
\end{theorem}

We prove this theorem in later sections.  Let us assume it for now, and use it to establish Theorem \ref{main}.  The basic idea (which is standard in the Stieltjes transform approach to the local semicircle law) is to use a truncated form of \eqref{bless}. Let $M_n, W_n, T, K$ be as in the above theorem.  By the triangle inequality, we may take $I = (-\infty,E)$ for some real number $E$; from the support of $\rho_\sc$, we may assume that $E \in [-2,2]$.  We may also take $T \gg \log^{100} n$ (say), as the claim is trivial otherwise.  

Let $T_1 \gg T/\log n \gg \log^{99} n$ be a quantity to be chosen later, and set $\eta_0 := T_1 / n$.  Applying Theorem \ref{wigo} and the union bound, we see that outside of an event of probability at most
\begin{equation}\label{noo}
 n^{O(1)} \exp(-c T_1^{-c}),
\end{equation}
one has
\begin{equation}\label{sawn}
|s_{W_n}(E + \sqrt{-1} \eta) - s_\sc(E + \sqrt{-1} \eta)| \ll \frac{T_1}{n\eta}
\end{equation}
for all values of $\eta$ between $\eta_0$ and $n^{100}$ which are integer multiples of $n^{-1000}$.  On the other hand, in this range one easily verifies that the functions $\eta \mapsto s_{W_n}(E+\sqrt{-1}\eta)$ and $\eta \mapsto s_\sc(E+\sqrt{-1}\eta)$ are Lipschitz with Lipschitz norm at most $O(n^{200})$ (say).  As a consequence, we conclude (after conditioning outside of the above exceptional event) that \eqref{sawn} holds for \emph{all} $\eta$ between $\eta_0$ and $n^{100}$.  

By conditioning on another event of probability at most \eqref{noo}, we may assume that all entries of $M_n$ are of size at most $O(n)$ (say).  Among other things, this implies  that all eigenvalues $\lambda_i(W_n)$ are (very crudely) of size at most $O(n^{20})$.

Since $\eta \geq \eta_0 = T_1/n$, we conclude from \eqref{sawn} and \eqref{explicit} that
$$ |s_{W_n}(E + \sqrt{-1} \eta)| \ll 1.$$
On the other hand, from \eqref{swn-def} one has
$$ \Im s_{W_n}(E + \sqrt{-1} \eta) = \frac{1}{n} \sum_{i=1}^n \frac{\eta}{|\lambda_i(W_n)-E|^2 + \eta^2}$$
and in particular
$$ \Im s_{W_n}(E + \sqrt{-1} \eta) \gg \frac{1}{n\eta} N_{[E-\eta,E+\eta]}.$$
We conclude that\footnote{One could also have used Proposition \ref{crudo} at this juncture.}
\begin{equation}\label{neeta}
 N_{[E-\eta,E+\eta]} \ll n \eta
\end{equation}
for all $\eta \geq \eta_0$ (note that this claim is trivial for $\eta \geq n^{100}$).  

Next, if we integrate \eqref{sawn} and use the triangle inequality, we observe that
\begin{equation}\label{oi}
\Re \int_{\eta_0}^{n^{100}} s_{W_n}(E + \sqrt{-1} \eta)\ d\eta =
\Re \int_{\eta_0}^{n^{100}} s_\sc(E + \sqrt{-1} \eta)\ d\eta + O\left( \frac{T_1 \log n}{n} \right).
\end{equation}
Let us now evaluate the left-hand side.  From the definition of the Stieltjes transform, we may rewrite it as
$$ \frac{1}{n} \sum_{i=1}^n \Arg( E + \sqrt{-1} \eta_0 - \lambda_i(W_n) ) - \Arg( E + \sqrt{-1} n^{100} - \lambda_i(W_n) ),$$
where $\Arg$ is the standard branch of the argument on the upper half-plane.

Since $E \in [-2,2]$ and $\lambda_i(W_n) = O(n^{20})$, we have
$$ \Arg( E + \sqrt{-1} n^{100} - \lambda_i(W_n) ) = \frac{\pi}{2} + O( n^{-50} )$$
(say).  Also, from elementary trigonometry one has
$$ \Arg( E + \sqrt{-1} \eta_0 - \lambda_i(W_n) ) = \pi 1_{\lambda_i(W_n) \geq E} + O\left( \frac{\eta_0}{|\lambda_i(W_n)-E| + \eta_0} \right).$$
We may therefore write the left-hand side of \eqref{oi} as
$$  \frac{\pi}{2} - \frac{1}{n} \pi N_{(-\infty,E)} + O\left( \frac{1}{n} \sum_{i=1}^n \frac{\eta_0}{|\lambda_i(W_n)-E| + \eta_0} \right) + O( n^{-50} )$$
(compare with \eqref{bless}).
On the other hand, from \eqref{neeta} and dyadic decomposition (recalling that $\lambda_i(W_n) = O(n^{20})$) one has
$$ \frac{1}{n} \sum_{i=1}^n \frac{\eta_0}{|\lambda_i(W_n)-E| + \eta_0} = O( \eta_0 \log n )$$
and thus
$$
\Re \int_{\eta_0}^{n^{100}} s_{W_n}(E + \sqrt{-1} \eta)\ d\eta =
\frac{\pi}{2} - \frac{1}{n} \pi N_{(-\infty,E)} + O\left( \frac{T_1 \log n}{n} \right).$$
A similar argument gives
$$
\Re \int_{\eta_0}^{n^{100}} s_\sc(E + \sqrt{-1} \eta)\ d\eta =
\frac{\pi}{2} - \pi \int_{-\infty}^E \rho_\sc(y)\ dy + O\left( \frac{T_1 \log n}{n} \right).$$
From \eqref{oi} we thus conclude that
$$ N_{(-\infty,E)} = n\int_{-\infty}^E \rho_\sc(y)\ dy + O( T_1 \log n ).$$
Choosing $T_1$ to be a small multiple of $T/\log n$ (and bounding $T_1^c$ from below by $T^{c'} - O(\log n)$ for some sufficiently small $c'>0$), we obtain Theorem \ref{main} as desired.

It remains to deduce Theorem \ref{wigo} from Proposition \ref{gack}.  This will be the objective of the next few sections.

\section{The moment method, and the Lindeberg strategy}\label{moment-sec}

Given a matrix $W_n = \frac{1}{\sqrt{n}} M_n$ and a complex number $z = E + \sqrt{-1} \eta$, define the quantity $A(W_n) = A(W_n,z)$ by the formula
$$ A(W_n) := n\eta (s_{W_n}(z) - s_\sc(z)).$$
This quantity describes the normalised deviation of the Stieltjes transform of $W_n$ from the semicircular law at $z$.  In this notation, Proposition \ref{gack} becomes the assertion that
\begin{equation}\label{awt-1}
 \P( |A(W_n)| \geq T ) \ll n^{O(1)} \exp( - c T )
\end{equation}
whenever $T>0$, $E \in \R$, and $\eta > 0$, when $M_n$ is drawn from GUE.  Similarly, Theorem \ref{wigo} becomes the assertion that
\begin{equation}\label{awt-2}
\P( |A(W_n)| \geq T ) \ll n^{O(1)} ( \exp( - c T^c ) + \exp( - c (n\eta)^c ) )
\end{equation}
whenever $T>0$, $E \in [-3,3]$, and $0 < \eta \ll n^{100}$, when $M_n$ is drawn from a Wigner matrix obeying Condition {\condo}, and with $\Re \xi_{ij}$ and $\Im \xi_{ij}$ having variance $1/2$ and third moment zero for $1 \leq i < j \leq n$.

To deduce \eqref{awt-2} from \eqref{awt-1} we will use the moment method combined with the Lindeberg exchange strategy; more specifically, we will show that a high moment $\E A(W_n)^k$ for some large even number $k$ (which one should think of, in practice, as comparable to $T$) is stable under the operation of replacing (the real or imaginary part of) one entry of $M_n$ (and its transpose) with another entry with a number of matching moments.  The Lindeberg exchange strategy is by now a standard tool in establishing universality properties for Wigner matrices \cite{TVlocal1}, \cite{EYY}, \cite{EYY2}; the main novelty here\footnote{Very recently \cite{ky0}, a similar application of the Lindeberg exchange strategy to a high moment of a spectral statistic was used to establish some related concentration results.  We thank Antti Knowles for bringing this preprint to our attention.}  is the application of that strategy to a high moment $\E A(W_n)^k$ (as opposed to a quantity such as $\E G( A(W_n) )$ for some smooth test function $G$).

Let us now make the strategy more precise.  Let us call two Wigner matrices $M_n, M'_n$ \emph{real-adjacent}, or \emph{adjacent} for short, if their respective atom variables $\xi_{ij}, \xi'_{ij}$ are equal except for a single choice of $(i,j)=(a,b)$ and its transpose $(i,j)=(b,a)$, and such that $\xi_{ab}, \xi'_{ab}$ either have identical real parts, or identical imaginary parts.  Thus, a Wigner matrix $M'_n$ adjacent to $M_n$ is formed by changing the real or imaginary part of a single entry of $M_n$ and its adjoint, leaving the other components of $M_n$ unchanged.   The main technical step is then to establish the following proposition.

\begin{proposition}[Stability of moments]\label{stab-mom}  Let $M_n, M'_n$ be two adjacent Wigner matrices obeying Condition {\condo}, whose moments match to order $m$ for some fixed $m = O(1)$.  Let $z = E + \sqrt{-1} \eta$ for some $E \in [-3,3]$ and $0 < \eta \ll n^{100}$, and set $W_n := \frac{1}{\sqrt{n}} M_n$ and $W'_n := \frac{1}{\sqrt{n}} M'_n$.  Then for any even integer $k \geq \log n$, one has
\begin{equation}\label{noodle}
 \E A(W_n)^k \leq \left(1 + O\left( \frac{1}{n^{(m+1)/2}} \right)\right) \E A(W'_n)^k + O(k)^k + O( n^{O(k)} \exp(-(n\eta)^c) ).
\end{equation}
\end{proposition}

Let us assume this proposition for now and establish Theorem \ref{wigo}.  Let $n, M_n, W_n, E, \eta, z, T$ be as in that theorem.  We may assume that $T \geq \log^{C_0} n$ (say) for some sufficiently large absolute constant $C_0$, as the claim is trivial otherwise; we may also assume that $T \leq \eta n$, since the claim follows from existing local semicircle laws (in particular, Corollary \ref{lsl}).  In particular, we may now assume that $T \leq n^{O(1)}$ and $\eta \geq \log^{C_0} n/n$.  Our task is now to show that
\begin{equation}\label{pawn}
 \P( |A(W_n)| \geq T ) \ll n^{O(1)} \exp( - c T^c ).
\end{equation}

On the other hand, if $M'_n$ is drawn from GUE and $W'_n := \frac{1}{\sqrt{n}} M'_n$, then from Proposition \ref{gack} one has
$$ \P( |A(W'_n)| \geq T ) \ll n^{O(1)} \exp( - c T )$$
for all $T>0$.  In particular, for any $k \geq \log n$, one has
\begin{equation}\label{spelunk}
\begin{split}
\E |A(W'_n)|^k &= \int_0^\infty \P( |A(W'_n)| \geq T ) k T^{k-1}\ dT \\
&\ll k n^{O(1)} \int_0^\infty e^{-cT} T^{k-1}\ dT \\
&\ll O(1)^k n^{O(1)} k! \\
&\ll O(k)^k
\end{split}
\end{equation}

We can replace $M'_n$ with $M_n$ in a sequence of $n^2$ exchanges from one Wigner matrix to a real-adjacent one; $n^2-n$ of these exchanges arise by swapping the real or imaginary part of an off-diagonal entry $\xi_{ij}$ of $M'_n$ (and its transpose $\xi_{ji}$) with the corresponding component of $M_n$, and $n$ of these exchanges arise by swapping a diagonal entry $\xi_{ii}$ of $M'_n$ with the corresponding entry of $M_n$.  We perform these exchanges in an arbitrary order.  By hypothesis, for the $n^2-n$ off-diagonal exchanges one has matching moments to order $m=3$, while for the diagonal exchanges one has matching moments to order $m=1$.  Let $M_n = M^0_n, M^1_n, \ldots, M^{n^2}_n = M'_n$ denote the sequence of exchanges from $M_n$ to $M^{n^2}_n$, and let $W^0_n,\ldots,W^{n^2}_n$ be the associated rescaled Wigner matrices.  By Proposition \ref{stab-mom} one has
$$ \E A(W^a_n)^k \leq (1 + O(\frac{1}{n^{(m_a+1)/2}} \E A(W^{a+1}_n)^k + O(k)^k + O( n^{O(k)} \exp(-(n\eta)^c) ),$$
for $0 \leq a < n^2$, where $m_a$ is equal to $3$ for $n^2-n$ choices of $a$ and equal to $1$ for $n$ choices of $a$.  Concatenating these bounds, we conclude that for any $k \geq \log n$ one has
$$ \E A(W_n)^k \leq O(1) \E A(W'_n)^k + O(n^2) O(k)^k + O( n^{O(k)} \exp(-(n\eta)^c) ).$$
In particular, from \eqref{spelunk} one has
$$ \E A(W_n)^k \ll O(k)^k + O( n^{O(k)} \exp(-(n\eta)^c) )$$
and hence by Markov's inequality
$$ \P(|A(W_n)| \geq T) \ll (\frac{O(k)}{T})^k + T^{-k} n^{O(k)} \exp(-(n\eta)^c).$$
If we set $k$ to be the largest even integer less than $T^{c_0}$ for some absolute constant $c_0$, and if $C_0$ is sufficiently large depending on $c_0$, we obtain \eqref{pawn} as desired, thanks to the assumptions $\log^{C_0} n \leq T \leq n\eta$.

\begin{remark} An inspection of the above argument reveals that we in fact have the slight refinement
$$  \P( |A(W_n)| \geq T ) \ll n^{O(1)} \exp( - c T )$$
in the regime $T \leq (n\eta)^c$, since in this regime we may take $k$ to be a small multiple of $T$ (rounded off to the nearest even integer, of course).  Unfortunately, this refinement does not appear to immediately offer any significant improvement to the conclusion of Theorem \ref{main}.
\end{remark}

It remains to establish Proposition \ref{stab-mom}.  This will be achieved in the next section.

\section{Stability of high moments}\label{stable-sec}

We now prove Proposition \ref{stab-mom}.  We introduce a definition:

\begin{definition}[Elementary matrix]  An \emph{elementary matrix} is a matrix which has one of the following forms
\begin{equation}\label{vform}
 V = e_a e_a^*, e_a e_b^* + e_b e_a^*, \sqrt{-1} e_a e_b^* - \sqrt{-1} e_b e_a^*
\end{equation}
with $1 \leq a,b \leq n$ distinct, where $e_1,\ldots,e_n$ is the standard basis of $\C^n$.
\end{definition}

As $M_n, M'_n$ are real-adjacent, one can write
$$ M_n = M_n^0 + \xi V; \quad M'_n = M_n^0 + \xi' V$$
for some elementary matrix $V$, some random matrix $M_n^0$, and some real random variables $\xi, \xi'$ independent of $M_n^0$ that match moments to $m^{\operatorname{th}}$ order and obey the exponential decay condition
\begin{equation}\label{exp-decay-again}
 \P( |\xi| \geq t^C ), \P( |\xi'| \geq t^C ) \leq e^{-t}
\end{equation}
for all $t \geq C'$ and some $C,C' > 0$.

We now recall some (deterministic) resolvent stability results concerning matrices of the form $M_n^0 + t V$.  Define the matrix norm $\|R \|_{(\infty,1)}$ of a $n \times n$ matrix $R = (R_{ij})_{1 \leq i,j \leq 1}$ by the formula
$$ \|R\|_{(\infty,1)} := \sup_{1 \leq i,j \leq n} |R_{ij}|.$$

\begin{proposition}[Stability of resolvent]\label{stabres}  Let $M_n^0$ be a Hermitian matrix, let $V$ be an elementary matrix, and let $t$ be a real number.  Let $z := E + \sqrt{-1} \eta$ be a complex number with $\eta>0$. Write
$$ R_t := (M_n^0 + tV - z)^{-1}$$
and suppose that
$$ |t| \|R_0\|_{(\infty,1)} = o(\sqrt{n}).$$
Then
$$ \|R_t\|_{(\infty,1)} = (1+o(1)) \|R_0\|_{(\infty,1)}.$$
Furthermore, if we set $s_t := \frac{1}{n} \tr R_t$, then we have the Taylor expansion
$$
 s_t = s_0 + \sum_{j=1}^m n^{-j/2} c_j t^j + O( n^{-(m+1)/2} |t|^{m+1} \|R_0\|_{(\infty,1)}^{m+1} \min( \|R_0\|_{(\infty,1)}, \frac{1}{n\eta} ) )
$$
for any fixed nonnegative $m=O(1)$, where the coefficients $c_j$ are independent of $t$ and obey the bounds
\begin{equation}\label{cj-bound}
|c_j| \ll \|R_0\|_{(\infty,1)}^{j} \min( \|R_0\|_{(\infty,1)}, \frac{1}{n\eta} ).
\end{equation}
for all $1 \leq j \leq m$.
\end{proposition}

\begin{proof}  See \cite[Lemma 12]{TVdetlaw} and \cite[Proposition 13]{TVdetlaw}. \end{proof}

Our objective is to establish \eqref{noodle}.  From Corollary \ref{rb} we see that
$$ \| R_\xi \|_{(\infty,1)} = O(1)$$
with probability $1-O(n^{O(1)} \exp(-(n\eta)^c)$, while from \eqref{exp-decay-again} we certainly have $\xi = o(\sqrt{n})$ with $1-O(n^{O(1)} \exp(-(n\eta)^c)$.   Hence by the first conclusion of Proposition \ref{stabres} (with $M_n^0$ and $V$ replaced with $M_n^0+\xi V$, and setting $t$ equal to $-\xi$) we have
\begin{equation}\label{mno}
 \| R_0 \|_{(\infty,1)} = O(1)
\end{equation}
with probability $1-O(n^{O(1)} \exp(-(n\eta)^c))$.  Using the crude bound $A(W_n) = O(n^{O(1)})$, we may thus condition $M_n^0$ to be fixed and obeying \eqref{mno}, since the contribution of the event where \eqref{mno} fails to $\E A(W_n)^k$ is $O(n^{O(k)} \exp(-(n\eta)^c))$.  

By Proposition \ref{stabres}, we thus see that whenever $\xi = o(\sqrt{n})$, one has
\begin{equation}\label{aww}
A(W_n) = A_0 + \sum_{j=1}^m a_j (\xi/\sqrt{n})^j + O( (|\xi|/\sqrt{n})^{m+1} )
\end{equation}
where the coefficients $A_0, a_j$ are deterministic (and in particular independent of $\xi, \xi'$, though they can depend on $\eta,n$), and $a_j$ obeys the bound $a_j = O(1)$.  

Suppose first that $|A_0| \leq k$.  Then one has
$$ |A(W_n)| \ll k $$
whenever $\xi = o(\sqrt{n})$, which gives a net contribution of $O(k)^k$ to $\E |A(W_n)|^k$; meanwhile, from \eqref{exp-decay-again}, the case when $\xi \gg \sqrt{n}$ contributes at most $O(n^{O(k)} \exp(-(n\eta)^c))$.  Thus we may assume that $|A_0| > k$.  Thus we have
$$
A(W_n) = A_0 \left(1 + \frac{1}{k} \left(\sum_{j=1}^m b_j (\xi/\sqrt{n})^j + O\left( (\xi/\sqrt{n})^{m+1} \right) \right) \right)$$
for some deterministic coefficients $b_1,\ldots,b_m = O(1)$, and assuming that $\xi = o(\sqrt{n})$.  Raising this to the $k^{\operatorname{th}}$ power (after using Taylor's theorem with remainder to expand $(1 + \frac{1}{k} x)^k$ to $m^{\operatorname{th}}$ order in the regime $x=o(1)$), we conclude that
$$
A(W_n)^k = A_0^k \left(1 + \sum_{j=1}^m d_j (\xi/\sqrt{n})^j + O\left( (|\xi|/\sqrt{n})^{m+1} \right) \right)$$
for some deterministic coefficients $d_1,\ldots,d_m = O(1)$ (which are allowed to depend on $k$), whenever $\xi = o(\sqrt{n})$.  Taking (conditional) expectations in $\xi$ (using \eqref{exp-decay-again} and the trivial bound $A(W_n) = O(n^{O(1)})$ to handle the tail event when $|\xi| \gg \sqrt{n}$) we conclude that
$$ \E(A(W_n)^k|M_n^0) = A_0^k \left(1 + \sum_{j=1}^m d_j n^{-j/2} \E \xi^j + O( n^{-(m+1)/2} ) \right) + O( n^{O(k)} \exp(-(n\eta)^c) ).$$
and thus
$$ \E A(W_n)^k = \E\left(A_0^k \left(1 + \sum_{j=1}^m d_j n^{-j/2} \E \xi^j + O( n^{-(m+1)/2} ) \right)\right) + O( n^{O(k)} \exp(-(n\eta)^c) ) + O(k)^k.$$
Similarly we have
$$ \E A(W'_n)^k = \E\left(A_0^k \left(1 + \sum_{j=1}^m d_j n^{-j/2} \E (\xi')^j + O( n^{-(m+1)/2} )\right)\right) + O( n^{O(k)} \exp(-(n\eta)^c) ) + O(k)^k.$$
Since $\xi$ and $\xi'$ match to order $k$, we obtain the claim.  This concludes the proof of Proposition \ref{stab-mom} and hence Theorem \ref{main}.

\begin{remark}\label{non-indep}  It is possible to adapt the above arguments to the case when $\Re \xi_{ij}$ and $\Im \xi_{ij}$ are not assumed to be independent.  The main new difficulty is that instead of swapping the real and imaginary parts of a single entry $\xi_{ab}$ of $M_n$ (and its transpose $\xi_{ba}$) separately, one has to swap them together.  This requires one to consider perturbations of the form
$$ M_n = M_n^0 + \xi_1 V_1 + \xi_2 V_2$$
where $V_1, V_2$ are two distinct elementary random variables, and $\xi_1, \xi_2$ are real random variables that are not necessarily independent and obeying the exponential decay hypothesis \eqref{exp-decay-again}.  However, it is possible to extend Proposition \ref{stabres} without much difficulty to the case of two-parameter perturbations and perform a similar argument to that given above.  We omit the details.
\end{remark}

\section{Extreme eigenvalues}\label{second-sec}

We now prove Theorem \ref{main-2}, by combining the arguments in previous sections with some ideas from \cite{EYY2} (and in particular, demonstrating a concentration of $\Im s_{W_n}(E + \sqrt{-1} \eta)$ that is better than $1/n\eta$ for some energy $E>2$).  By symmetry, it suffices to prove the bound for $\lambda_1(W_n)$.  We may of course assume that $n$ is large.

By standard large deviation estimates, one has
$$ \P( \lambda_1(W_n) \geq E ) \ll \exp( - c n^c \log E )$$
for any $E \geq 3$; see\footnote{One could also use the earlier estimates in \cite{meckes} or \cite{akv}; see also \cite{AGZ} for more discussion.} \cite[Lemma 7.2]{EYY}.  This already deals with the case when $n^{2/3} \leq T \leq n^{100}$ (say), and the case $T>n^{100}$ can be handled by crudely bounding $\lambda_1(W_n)$ by, say, the Frobenius norm of $W_n$ and using Condition \condo.  Thus we may restrict attention to the regime $T \leq n^{2/3}$, and show that
$$ \P( 2 + n^{-2/3} T \leq \lambda_1(W_n) \leq 3 ) \ll n^{O(1)} \exp( - c T^c ).$$
We may assume that $T \geq \log^{C_0} n$ for some suitably large absolute constant $C_0$, as the claim is trivial otherwise.

Suppose that $\lambda_1(W_n)$ was in the interval $[2 + n^{-2/3} T, 3]$.  Set $\eta := n^{-2/3}$, and let $B(W_n)$ denote the quantity
$$ B(W_n) := n \eta \Im s_{W_n}(E + \sqrt{-1} \eta).$$
From the identity
\begin{equation}\label{swain}
B(W_n) = \sum_{i=1}^n \frac{\eta^2}{|\lambda_i(W_n)-E|^2 + \eta^2}
\end{equation}
we conclude in particular that
$$ B(W_n) \geq \frac{1}{10}$$
where $E$ is the closest multiple of $n^{-2/3}$ in $[2+n^{-2/3} T, 3]$ to $\lambda_1(W_n)$.  Thus, by the union bound, it will suffice to show that
\begin{equation}\label{bween}
 \P( B(W_n) \geq \frac{1}{10} ) \ll n^{O(1)} \exp(-c T^c)
\end{equation}
for any fixed $E \in [2+n^{-2/3} T, 3]$.

Let $M'_n$ be drawn from GUE, and set $W'_n := \frac{1}{\sqrt{n}} M'_n$.  By Theorem \ref{conc-extreme}, we have
$$ \lambda_1(W'_n) \leq 2 + n^{-2/3} T/2$$
outside of an event of probability $O( \exp(-c T^{3/2}) )$; in particular, we have
\begin{equation}\label{e1}
 N_{[E - n^{-2/3} T/2, E + n^{-2/3} T/2]}(W'_n) = 0
 \end{equation}
outside of this event.

Also, from Corollary \ref{conc} and the union bound we see that outside of an event of probability $O( n^{O(1)} \exp(-cT) )$, one has
$$ N_I(W'_n) \leq n\int_I \rho_{sc}(y) \ dy + O( T^{0.1} )$$
(say) for all intervals $I$.  In particular, outside of this event, we have
\begin{equation}\label{e2}
\begin{split}
N_{ [E - 2^k n^{-2/3} T, E + 2^k n^{2/3} T] }(W'_n) &\leq
n \int_{2 - 2^k n^{-2/3} T}^2 \rho_{sc}(y)\ dy + O(T^{0.1})\\
&\ll 2^{3k/2} T^{3/2}
\end{split}
\end{equation}
for all $k \ge 1$, using the bound $\rho_{sc}(y) = O( (2-y)^{1/2} )$ when $y < 2$.  

From \eqref{e1}, \eqref{e2}, \eqref{swain}, and dyadic decomposition one easily establishes that
$$ B(W'_n) \ll \frac{1}{T^{1/2}}$$
outside of an event of probability $O( n^{O(1)} \exp(-cT^c) )$.

Let $\log n \leq k \leq n^{0.01}$ be an integer to be chosen later.  Since we may trivially bound $\Im s_{W_n}(E + \sqrt{-1} \eta)$ by $n^{O(1)}$, we conclude that
\begin{equation}\label{bwink}
\E B(W'_n)^k \ll O\left(\frac{1}{T}\right)^{k/2} + n^{O(k)} \exp(-cT^c).
\end{equation}

We claim the following stability result for $\E B(W_n)^k$, analogous to Proposition \ref{stab-mom}:

\begin{proposition}[Stability of moments]\label{stab-mom-2}  Let $M_n, M'_n$ be two adjacent Wigner matrices obeying Condition {\condo}, whose moments match to order $m$ for some fixed $m = O(1)$.  Set $W_n := \frac{1}{\sqrt{n}} M_n$ and $W'_n := \frac{1}{\sqrt{n}} M_n$.  Then for any integer $\log n \leq k \leq n^{0.1}$, one has
\begin{equation}\label{noodle-2}
 \E B(W_n)^k \leq (1 + O( (k/\sqrt{n})^{m+1} )) \E B(W'_n)^k + O(100^{-k}) + O( n^{O(k)} \exp(-cT^c) ).
\end{equation}
\end{proposition}

Applying this proposition $n^2-n$ times with $m=4$ and $n$ times with $m=2$ we conclude that
$$ \E B(W_n)^k \ll (1 + O(k^5/n^{5/2}))^{n^2-n} (1 + O(k^3/n^{3/2}))^{n} ( \E B(W'_n)^k + O(n^{O(1)} 100^{-k}) + O( n^{O(k)} \exp(-cT^c) ) )$$
and thus (using \eqref{bwink} and the hypothesis $k \leq n^{0.01}$)
$$ \E B(W_n)^k \ll n^{O(1)} 100^{-k} + n^{O(k)} \exp(-cT^c) .$$
The desired claim \eqref{bween} then follows from Markov's inequality by taking $k = T^{c_0}$ for some sufficiently small $c_0>0$ (and assuming $C_0$ sufficiently large depending on $c_0>0$).

It remains to establish Proposition \ref{stab-mom-2}.  As in the previous section, we write
$$ M_n = M_n^0 + \xi V; \quad M'_n = M_n^0 + \xi' V$$
for some elementary matrix $V$, some random matrix $M_n^0$, and some real random variables $\xi, \xi'$ independent of $M_n^0$ that match moments to $m^{\operatorname{th}}$ order and obey the exponential decay condition \eqref{exp-decay-again}.  Arguing exactly as before, we may condition $M_n^0$ to be a deterministic matrix for which
$$ \| R_0 \|_{(\infty,1)} = O(1).$$
Using Proposition \ref{stabres} as before, we see that
$$ B(W_n) = B_0 + \sum_{j=1}^m a_j (\xi/\sqrt{n})^j + O( (|\xi|/\sqrt{n})^{m+1} )$$
for some deterministic coefficients $B_0$ and $a_j = O(1)$, whenever $\xi = o(\sqrt{n})$.

Suppose first that $|B_0| \leq 1/200$.  Then one has $|B(W_n)| \leq 1/100$ whenever $\xi = o(\sqrt{n})$, and so this case contributes $O( 100^{-k} ) + O( n^{O(k)} \exp(-cn^c) )$ to \eqref{noodle-2}, which is acceptable.  Thus we may restrict attention to the case when $|B_0| > 1/200$.  Then we may write
$$ B(W_n) = B_0 \left( 1 + \sum_{j=1}^m b_j (\xi/\sqrt{n})^j + O( (|\xi|/\sqrt{n})^{m+1} ) \right)$$
whenever $\xi = o(\sqrt{n})$, where the $b_j=O(1)$ are deterministic coefficients.

Suppose now that $\xi = O(n^{0.3})$.  Since $k \leq n^{0.01}$, we may perform a Taylor expansion of $(1+x)^k$ to order $m$ for $x = O(n^{-0.2})$ and conclude that
$$ B(W_n)^k = B_0^k \left( 1 + \sum_{j=1}^m c_j (k \xi/\sqrt{n})^j + O( (k |\xi|/\sqrt{n})^{m+1} ) \right)$$
in this regime, where the $c_j=O(1)$ are deterministic coefficients (which are allowed to depend in $k$).  Taking expectations as in the preceding section, and using \eqref{exp-decay-again} to handle those $\xi$ with $|\xi| \geq n^{0.3}$, we conclude that
\begin{align*}
\E B(W_n)^k &= \E\left(B_0^k \left( 1 + \sum_{j=1}^m c_j k^j n^{-j/2} \E \xi^j + O( (k/\sqrt{n})^{m+1} )\right)\right)\\
&\quad + O( n^{O(k)} \exp(-(n\eta)^c) ) + O( 100^{-k} ),
\end{align*}
and similarly for $\E B(W'_n)^k$; and the claim follows from the matching moments hypothesis.

\begin{remark} As in Remark \ref{non-indep}, it is possible to extend these arguments to the case when $\Re(\xi_{ij})$ and $\Im(\xi_{ij})$ are not independent; we leave the details to the interested reader.
\end{remark}

\begin{remark} \label{remark:bestT} Note that when one has four matching moments rather than three, the error terms are more favorable by a factor of $\sqrt{n}$, giving some additional room to vary the parameters of the argument by small powers of $n$.  Because of this, it is possible to modify the proof of Theorem \ref{wigo} to conclude in this case that
$$  \P( |A(W_n)| \geq T ) \ll n^{O(1)} \exp( - c T )$$
in the regime $0 < T \leq n^c$ for a sufficiently small $c$.  This is achieved by arguing as in this section, except that one allows the resolvent $\|R_0\|_{(\infty,1)}$ to be as large as $O(n^c)$ rather than $O(1)$ in order to keep the failure probability bounded by $O(n^{O(1)} \exp(-n^c))$ rather than $O( n^{O(1)} \exp(-(n\eta)^c) )$.  We omit the details.  As a consequence, we can sharpen the conclusion of Theorem \ref{main} to
$$ \P\left( |N_I(W_n) - n \int_I \rho_\sc(y)\ dy| \geq T\right) \ll n^{O(1)} \exp( - c T )$$
when $0 < T \leq n^c$ and $M_n$ matches moments with GUE to fourth order off the diagonal and second order on the diagonal.
\end{remark}

\appendix

\section{Local semicircle law}

In this appendix we establish some preliminary local semicircle law estimates, following the treatment in \cite{EYY} and \cite{TVlocal1}.  As the methods used here are now standard, and the results very close to those in \cite{EYY} and \cite{TVlocal1}, we shall be somewhat brief in our treatment.

We first recall a concentration estimate of Hanson and Wright \cite{hanson}.

\begin{proposition}[Concentration of quadratic forms]\label{hanson}  Let $X \in \C^n$ be a vector of independent random variables $\xi_1,\ldots, \xi_n$ of mean zero and variance $\sigma^2$, obeying the uniform subexponential decay bound
$$ \P( |\xi_i|\geq t^C \sigma ) \leq e^{-t}$$
for all $t \geq C'$ and $1 \leq i \leq n$, and some $C, C' >0$ independent of $n$.  Let $A$ be an $n \times n$ matrix.  Then for any $T>0$, one has
$$ \P( |X^* A X - \sigma^2 \tr A| \geq T \sigma^2 (\tr(A^* A))^{1/2} ) \ll \exp( - c T^c ).$$
Thus
$$ X^* A X = \sigma^2 ( \tr A + O( T \tr(A^* A))^{1/2} )$$
outside of an event of probability $O( \exp(-cT^c) )$.
\end{proposition}

\begin{proof} See \cite[Lemma B.1]{EYY}.  (Note that a factor of $\sigma$ is missing from the statement of the exponential decay hypothesis in the lemma as stated in \cite{EYY}, which is needed in order to reduce to the $\sigma=1$ case.)
\end{proof}

\begin{corollary}[Distance between a random vector and a subspace]\label{dust}  Let $X$ and $\sigma$ be as in Proposition \ref{hanson}, and let $V$ be a $d$-dimensional complex subspace of $\C^n$.  Let $\pi_V$ be the orthogonal projection to $V$.  Then one has
$$ 0.9 d \sigma^2 \leq \| \pi_V(X) \|^2 \leq 1.1 d \sigma^2$$
outside of an event of probability $O( \exp( - c d^c ) )$.
\end{corollary}

\begin{proof}  Apply the preceding proposition with $A := \pi_V$ (so $\tr A = \tr A^* A = d$) and $T := d^{1/2} / 10$.
\end{proof}

\begin{remark}
We can also use Talagrand's inequality as in \cite{TVlocal1}, combining with a truncation argument (to bound each entries by some properly chosen 
quantity $K$). In the case when the atom variables have very fast decay (such as sub-gaussian) or bounded (such as Bernoulli),  this calculation 
will actually lead to a decent bound on the value of $c$ in Theorem \ref{main}. 
\end{remark} 

We can now establish a crude upper bound on the counting function $N_I$ of a Wigner matrix.

\begin{proposition}[Crude upper bound]\label{crudo}  Let $M_n$ be a Wigner matrix obeying Condition \condo, and let $W_n := \frac{1}{\sqrt{n}} M_n$.  Then for any interval $I$, one has
$$ N_I(W_n) = O(n |I|)$$
outside of an event of probability $O( n^{O(1)} \exp( - c (n|I|)^c ) )$.
\end{proposition}

\begin{proof}  Fix $I$, which we write as $I = [E-\eta,E+\eta]$.  Suppose that
\begin{equation}\label{niwnc}
 N_I(W_n) \geq C n \eta
\end{equation}
for some sufficiently large absolute constant $C$ to be chosen later.  We will show that this leads to a contradiction outside of an event of probability $O( n^{O(1)} \exp( - c (n\eta)^c )$.

From the identity
$$ \Im s_{W_n}(E + \sqrt{-1} \eta) = \frac{1}{n} \sum_{i=1}^n \frac{\eta}{|\lambda_i(W_n) - E|^2 + \eta^2}$$
and \eqref{niwnc}, we see that
$$ \Im s_{W_n}(E + \sqrt{-1} \eta) \gg C.$$
On the other hand, we can write the Stieltjes transform $s_{W_n}$ in terms of the coefficients $R_{ij}$ of the resolvent as
$$ s_{W_n}(E + \sqrt{-1} \eta) = \frac{1}{n} \sum_{i=1}^n R_{ii}(E+\sqrt{-1} \eta).$$
Thus, by the pigeonhole principle, we have
$$ \Im R_{ii}(E + \sqrt{-1} \eta) \gg C$$
for some $1 \leq i\leq n$.  By symmetry (and conceding a factor of $n$ in the failure probability estimates) we may take $i=n$.

Now, a standard Schur complement computation (see e.g. \cite[Lemma 42]{TVlocal1}) shows that
\begin{equation}\label{shur}
 R(z)_{nn} = \frac{1}{\frac{1}{\sqrt{n}} \xi_{nn} - z - X^* R^{(n)}(z) X}
 \end{equation}
where $R^{(n)}(z) = (W_n^{(n)}-z)^{-1}$ is the resolvent corresponding to the $n-1 \times n-1$ matrix $W_n^{(n)}$ formed by removing the $n^{\operatorname{th}}$ row and column from $W_n$, $\xi_{nn}$ is the bottom right entry of $M_n$, and $X$ is the rightmost column of $W_n$ (after removing the bottom entry $\frac{1}{\sqrt{n}} \xi_{nn}$).  In particular, using the trivial bound $|\Im \frac{1}{z}| \leq \frac{1}{|\Im z|}$, we conclude that
$$ \Im R_{nn}(E + \sqrt{-1} \eta) \leq \frac{1}{\eta + \Im X^* R^{(n)}(E + \sqrt{-1} \eta) X} \leq \frac{1}{ \Im X^* R^{(n)}(z) X }$$
and thus
$$ \Im X^* R^{(n)}(E + \sqrt{-1} \eta) X \ll C^{-1}.$$
Now, by the Cauchy interlacing law, $W_n^{(n)}$ has $\gg Cn\eta$ consecutive eigenvalues in $I$.  There are $O(n^2)$ possibilities for the starting and ending index of these eigenvalues.  If we let $V$ be the space spanned by the corresponding eigenvectors, then $\dim(V) \gg Cn \eta$, and from the spectral theorem we see that
$$ \Im X^* R^{(n)}(E + \sqrt{-1} \eta) X \gg \|\pi_V(X)\|^2 / \eta$$
and thus
$$ \| \pi_V(X)\|^2 \ll \frac{1}{C} \eta.$$
On the other hand, from \eqref{dust} we see that
$$ \| \pi_V(X)\|^2 \gg C \eta$$
outside of an event of probability $O( \exp( - c (n\eta)^c ) )$.  If $C$ is sufficiently large, the claim follows.
\end{proof}

This gives rise to a self-consistent equation:

\begin{proposition}[Self-consistent equation]\label{sc}  Let $M_n$ be a Wigner matrix obeying Condition \condo, and let $W_n := \frac{1}{\sqrt{n}} M_n$.  Then for any $z = E + \sqrt{-1} \eta$ with $E = O(1)$ and $0 < \eta \ll n^{100}$, and all $1 \leq i \leq n$, one has
$$ R(z)_{ii} = -\frac{1}{s_{W_n}(z) + z + o(1) }$$
outside of an event of probability $O( n^{O(1)} \exp( - c(n\eta)^c ) )$.  In particular, by the union bound, we have
\begin{equation}\label{self-con}
 s_{W_n}(z) = -\frac{1}{s_{W_n}(z) + z + o(1)}
\end{equation}
outside of an event of probability $O( n^{O(1)} \exp( - c(n\eta)^c ) )$.
\end{proposition}

\begin{proof}  We can assume that $n\eta \geq \log^{100} n$ (say), as the claim is trivial otherwise.  By symmetry, it will suffice to establish
$$ R(z)_{nn} = -\frac{1}{s_{W_n}(z) + z + o(1) }$$
outside of an event of probability $O( n^{O(1)} \exp( - c(n\eta)^c ) )$.  By \eqref{shur}, this statement is equivalent to
$$
X^* R^{(n)}(z) X - \frac{1}{\sqrt{n}} \xi_{nn} = s_{W_n}(z) + o(1).$$
By Condition \condo, one has $\frac{1}{\sqrt{n}} \xi_{nn} = o(1)$ outside of an event of probability $O( \exp(-cn^c) )$, which is certainly acceptable; so our task is now to show that
\begin{equation}\label{xrx}
X^* R^{(n)}(z) X = s_{W_n}(z) + o(1)
\end{equation}
outside of an event of probability $O( n^{O(1)} \exp( - c(n\eta)^c ) )$.  

From the Cauchy interlacing law (cf. \cite[\S 5.2]{TVlocal1}) we know that
$$ \frac{1}{n} \tr R^{(n)}(z) = s_{W_n}(z) + o(1).$$
Also,
\begin{equation}\label{trr}
 \tr R^{(n)}(z)^* R^{(n)}(z) = \sum_{i=1}^{n-1} \frac{1}{|\lambda_i(W_n^{(n)})-E|^2 + \eta^2}.
 \end{equation}
By Proposition \ref{crudo} and the union bound, we may assume outside of an event of probability $O( n^{O(1)} \exp( - c(n\eta)^c ) )$, one has
$$ N_I(W_n) \ll n |I|$$
for all intervals $I$ of width at least $\eta$ centered at $E$.  By interlacing, we may also conclude
$$ N_I(W_n^{(n)}) \ll n |I|$$
for such intervals.  Inserting this bound into \eqref{trr}, we conclude that
\begin{equation}\label{tral}
\tr R^{(n)}(z)^* R^{(n)}(z) \ll \frac{n}{\eta}.
\end{equation}
If we then apply Proposition \ref{hanson} with $T := (n\eta)^{1/4}$ (say), using the hypothesis that $n\eta \geq \log^{100} n$ (so that $1/(n\eta)^c = o(1)$ for any $c>0$) we conclude \eqref{xrx} outside of an event of order $O( n^{O(1)} \exp( - c(n\eta)^c ) )$, as required.
\end{proof}

We can combine this proposition with a standard stability analysis of the self-consistent equation \eqref{self-con} to conclude a crude version of the local semicircle law:

\begin{corollary}[Local semicircle law]\label{lsl}  Let $M_n$ be a Wigner matrix obeying Condition \condo, and let $W_n := \frac{1}{\sqrt{n}} M_n$.  Then for any $z = E + \sqrt{-1} \eta$ with $E = O(1)$ and $0 < \eta \ll n^{100}$, and all $1 \leq i \leq n$, one has
\begin{equation}\label{swoon}
 s_{W_n}(z) = s_\sc(z) + o(1)
\end{equation}
and
\begin{equation}\label{roon}
 R(z)_{ii} = s_\sc(z) + o(1)
 \end{equation}
outside of an event with probability $O( n^{O(1)} \exp( - c(n\eta)^c ) )$.
\end{corollary}

We note that this corollary is essentially \cite[Theorem 3.1]{EYY2}; in the statement of the result in \cite{EYY2} the additional constraint $\eta \geq \log^{C \log \log n}/n$ for some constant $C$ is imposed, but this constraint is not actually used in the proof, at least if one is not concerned with obtaining the best possible bounds for the $o(1)$ error terms.  For the convenience of the reader, we sketch the proof of this corollary below.

\begin{proof}  As before we may assume that $\eta \geq \log^{100} n/n$; we may also assume that $n$ is large.  By Proposition \ref{sc}, we may assume that \eqref{self-con} holds.  

Let us first dispose of the case when $\eta$ is large, say $\eta \geq 100$.  In this case, the imaginary part of $s_{W_n}(z)+z+o(1)$ is at least $100-o(1)$, and hence by \eqref{self-con} one has $|s_{W_n}(z)| \leq 1/100 + o(1)$; inserting this back into \eqref{self-con} (and using \eqref{explicit}) one obtains $|s_{W_n}(z) - s_\sc(z)| \leq 1/10$ (say).  One can then deduce \eqref{swoon} from \eqref{self-con} (and \eqref{sce}) by a routine application of the contraction mapping theorem.

Henceforth we assume that $\eta < 100$, so that $z=O(1)$.  Then equation \eqref{self-con} already implies that 
$s_{W_n}(z)=O(1)$, since \eqref{self-con} cannot hold if $|s_{W_n}(z)|$ is too large.  We may thus multiply out the denominator and conclude that
$$ s_{W_n}(z)^2 + z s_{W_n}(z) + 1 = o(1).$$
Since the two solutions to the quadratic equation $s^2 + zs + 1 = 0$ are $s = s_\sc(z)$ and $s = -z-s_\sc(z)$, we conclude that
$$ s_{W_n}(z) = s_\sc(z)+o(1) \hbox{ or } s_{W_n}(z) = -z-s_\sc(z)+o(1)$$
outside of an event with probability $O( n^{O(1)} \exp( - c(n\eta)^c ) )$.

We apply this fact with $z$ replaced by an arbitrary complex numbers $\zeta$ with $\Re(\zeta) = O(1)$ and $\eta \leq \Im(\zeta) \ll 1$, and whose real and imaginary parts are multiples of $n^{-100}$ (say).  By the union bound, the probability of the failure event is still $O( n^{O(1)} \exp( - c(n\eta)^c ) )$.  We may then remove the latter hypotheses using the fact that $s_{W_n}$ and $s_\sc$ have Lipschitz constant $O(n)$ in this region, and conclude that outside of an event of probability $O( n^{O(1)} \exp( - c(n\eta)^c ) )$, one has
\begin{equation}\label{swnc}
 s_{W_n}(\zeta) = s_\sc(\zeta)+o(1) \hbox{ or } s_{W_n}(\zeta) = -\zeta-s_\sc(\zeta)+o(1)
\end{equation}
for \emph{all} $\zeta$ with $\Re(\zeta) = O(1)$ and $\eta \leq \Im(\zeta) \ll 1$. On the other hand, if one has $\Im(\zeta) \geq c$ for some absolute constant $c>0$, then the second possibility in \eqref{swnc} cannot occur for $n$ large enough, because $s_{W_n}(\zeta)$ necessarily has positive imaginary part.  A continuity argument then shows that the first option in \eqref{swnc} holds for all $\zeta$ in the indicated region\footnote{When $\zeta$ approaches the edges $\pm 2$ of the spectrum, thus $\zeta = \pm 2 + o(1)$, the two options in \eqref{swnc} begin to overlap, but in that regime one can deduce the first option from the second (with a slightly worse $o(1)$ error) and so the claim made in the text is still valid.}.  This gives \eqref{swoon}.  Among other things, this shows that $|s_{W_n}(z) + z| \gg 1$ (thanks to \eqref{sce}), and then from \eqref{sce} and the second part of Proposition \ref{sc} we obtain \eqref{roon}.
\end{proof}

For our applications, we will also need bounds on the coefficient norm
$$ \| R(z) \|_{(\infty,1)} := \sup_{1 \leq i,j \leq n} |R(z)_{ij}|$$
of the resolvent.

\begin{corollary}[Resolvent bound]\label{rb}  Let $M_n$ be a Wigner matrix obeying Condition \condo, and let $W_n := \frac{1}{\sqrt{n}} M_n$.  Then for any $z = E + \sqrt{-1} \eta$ with $E = O(1)$ and $0 < \eta \ll n^{100}$, one has
\begin{equation}\label{ro2}
 \|R(z)\|_{(\infty,1)} = O(1)
 \end{equation}
outside of an event with probability $O( n^{O(1)} \exp( - c(n\eta)^c ) )$.
\end{corollary}

\begin{proof} Again, we may assume $\eta > \log^{100} n/n$.  By the union bound, it suffices to show for each $1 \leq i, j \leq n$ that
$$
 |R(z)_{ij}| = O(1)$$
outside of an event with probability $O( n^{O(1)} \exp( - c(n\eta)^c ) )$.  In the diagonal case $i=j$, this follows directly from \eqref{roon}, so suppose that $i \neq j$.  In this case, we may use the Schur complement identity
$$ R(z)_{ij} = - R(z)_{ii} R^{(i)}(z)_{jj} K_{ij}^{(ij)}$$
where $R^{(i)}(z)$ is the resolvent associated to the $n-1 \times n-1$ matrix $W_n^{(i)}$ formed by removing the $i^{\operatorname{th}}$ row and column from $W_n$, and $K_{ij}^{(ij)}$ is the quantity
$$ K_{ij}^{(ij)} = \frac{1}{\sqrt{n}} \zeta_{ij} - X_i^* (W_n^{(ij)} - z)^{-1} X_j,$$
$\zeta_{ij}$ is the $ij$ coefficient of $W_n$, $W_n^{(ij)}$ is the $n-2 \times n-2$ matrix formed by removing the $i^{\operatorname{th}}$ and $j^{\operatorname{th}}$ rows and columns from $W_n$, and $X_i, X_j \in \C^{n-2}$ are the $i^{\operatorname{th}}$ and $j^{\operatorname{th}}$ columns of $W_n$, after removing the $i^{\operatorname{th}}$ and $j^{\operatorname{th}}$ rows.  See \cite[Lemma 4.2]{EYY} for a proof of this identity.  From \eqref{roon} applied to both the original Wigner matrix $W_n$ and the minor $W_n^{(i)}$ (which is essentially also a Wigner matrix, up to an easily manageable multiplicative factor of $\frac{\sqrt{n-1}}{\sqrt{n}}$) we see that $R(z)_{ii}=O(1)$ and $R^{(i)}(z)_{jj} = O(1)$ outside of an event of probability $O( n^{O(1)} \exp( - c(n\eta)^c ) )$, so it suffices to obtain the bound $K_{ij}^{(ij)} = O(1)$ outside of a similar event.  But from Condition \condo, one has $\frac{1}{\sqrt{n}} \zeta_{ij} = O(1)$ outside of an event of probability $O(\exp(-n^c))$, which is certainly acceptable, so it suffices to show that
$$ X_i^* (W_n^{(ij)} - z)^{-1} X_j = O(1)$$
outside of an event of probability $O( n^{O(1)} \exp( - c(n\eta)^c ) )$.    But by Proposition \ref{hanson} (viewing the $n-2 \times n-2$ matrix $(W_n^{(ij)} -z)^{-1}$ as the upper-right block of a nilpotent $2(n-2) \times 2(n-2)$ matrix, and concatenating $X_i$ and $X_j$ together), one has
$$ X_i^* (W_n^{(ij)} - z)^{-1} X_j = O( \frac{1}{n} T (\tr(((W_n^{(ij)} - z)^{-1})^* (W_n^{(ij)} - z)^{-1}))^{1/2} )$$
outside of an event of probability $O( \exp(-cT^c) )$, for any $T>0$.  But by repeating the derivation of \eqref{tral}, one has
$$ \tr(((W_n^{(ij)} - z)^{-1})^* (W_n^{(ij)} - z)^{-1}) = O( \frac{n}{\eta}).$$
If one then sets $T = O( \sqrt{n\eta} )$, one obtains the claim.
\end{proof}

We remark that the above argument in fact shows that we may improve the bound $R(z)_{ij}=O(1)$ to $R(z)_{ij} = O( \frac{1}{(n\eta)^{1/2-\delta}})$ for any fixed $\delta>0$; compare with \cite[Theorem 3.1]{EYY2}.  However, this improvement is not used in this paper.

\end{document}